\documentclass[12pt, leqno]{amsart}
\usepackage{graphicx}
\usepackage{amsfonts,delarray,amssymb,amsmath,amsthm,a4,a4wide}
\usepackage{latexsym}
\usepackage{epsfig}
\usepackage{color}

\newtheorem{thm}{Theorem}[section]

\newtheorem{lem}[thm]{Lemma}
\newtheorem{prop}[thm]{Proposition}
\theoremstyle{definition}

\newtheorem{rem}[thm]{Remark}
\numberwithin{equation}{section}
\newcommand{\norm}[1]{\left\Vert#1\right\Vert}
\newcommand{\abs}[1]{\left\vert#1\right\vert}

\newcommand{\R}{\mathbb R}

\newcommand{\e}{\varepsilon}
\newcommand{\eps}{\varepsilon}

\newcommand{\p}{\partial}

\newcommand{\comment}[1]{}

\newenvironment{myindentpar}[1]%
{\begin{list}{}%
         {\setlength{\leftmargin}{#1}}%
         \item[]%
}
{\end{list}}

\begin{document}

\title[The prescribed affine mean curvature and Abreu's equations]{Global second derivative estimates for the second boundary value problem of the prescribed affine mean curvature and Abreu's equations}
\author{Nam Q. Le}
\address{Department of Mathematics, Columbia University, New York, NY 10027, USA\\ and School of Engineering, Tan Tao University, Long An, Vietnam}
\email{\tt  namle@math.columbia.edu, nam.le@ttu.edu.vn}
\date{February 10, 2012}


\begin{abstract}
In this paper we prove the global second derivative estimates for the second boundary value problem of the prescribed affine mean curvature equation where the affine mean curvature is only assumed to be in $L^{p}$. Our result extends previous result by Trudinger and Wang in the case of globally bounded affine mean curvature and also covers Abreu's equation.

\end{abstract}
\maketitle

\section{Introduction and main results}
This paper is concerned with the global second derivative estimates for the second boundary value problem of the prescribed affine mean curvature equation. This equation is a fourth order fully nonlinear partial differential equation of the form
\begin{equation}
\label{AMCE}
L[u]:= U^{ij}w_{ij} =f\quad \text{in}~\Omega,
\end{equation}
where $u$ is a locally uniformly convex function defined in $\Omega\subset\R^{n}$, and throughout, $U = (U^{ij})$ is the matrix of cofactors of the 
Hessian matrix $D^{2}u$ of the convex function $u$, 
i.e., $$U = (\det D^{2} u) (D^{2} u)^{-1},$$and
\begin{equation}
w = [\det D^{2} u]^{\theta-1},~\theta<\frac{1}{n}.
\label{wu}
\end{equation}
The second boundary value problem for (\ref{AMCE}) is the Dirichlet problem for the system (\ref{AMCE})-(\ref{wu}), that is to prescribe
\begin{equation}
\label{SBV}
u=\varphi,~~~ w=\psi~~~\text{on}~\p\Omega.
\end{equation}

The problem (\ref{AMCE})-(\ref{SBV}) with $\theta=\frac{1}{n +2}$ was introduced by Trudinger and Wang in their investigation of the affine Plateau problem \cite{TW1} in affine geometry. In this context, the quantity
$$H_{A}[u]:= -\frac{1}{n+1} L[u]$$
is the {\it affine mean curvature} of the graph of $u$ \cite{Bl, NS, Si}. In particular, equation (\ref{AMCE}) with $f\equiv 0$ corresponds to the {\it affine maximal surface equation} \cite{TW}. The global second derivative estimates for (\ref{AMCE})-(\ref{SBV}) with $\theta\in (0, \frac{1}{n}]$ were established by Trudinger and Wang in \cite{TW2} under the assumption that $f\in L^{\infty}(\Omega)$. In particular, they proved
\begin{thm}(\cite[Theorem 1.2]{TW2})
\label{mainTW}
Let $\Omega$ be a uniformly convex domain in $\R^{n}$, with $\p\Omega\in C^{3,1}$, $\varphi \in C^{3,1}(\overline{\Omega}), \psi\in C^{3,1}(\overline{\Omega})$, $\inf_{\Omega}\psi>0$, $f\in L^{\infty}(\Omega)$, $f\leq 0$ and $\theta\in (0, \frac{1}{n}]$. Then, for uniformly convex solution $u\in C^{4}(\overline{\Omega})$ of (\ref{AMCE})-(\ref{SBV}), we have the estimate
\begin{equation}
\label{w4p}
\norm{u}_{W^{4,p}(\Omega)}\leq C,~~\text{for all}~1<p<\infty,
\end{equation}
where $C$ depends on $n, p, \theta, \p\Omega$, $\norm{f}_{L^{\infty}(\Omega)}$, $\norm{\varphi}_{C^{4}(\overline{\Omega})}, \norm{\psi}_{C^{4}(\overline{\Omega})}$, and $\inf_{\Omega} \psi.$
\end{thm}
It is very natural to ask if the estimate (\ref{w4p}) holds when $f$ is only assumed to be in $L^{p}$. In this paper, we answer this question for the case $p> n$. Our first main theorem provides global derivative estimates for solutions of the prescribed affine mean curvature equation.
\begin{thm}
\label{mainthm}
Let $p> n$ and let $\Omega$ be a uniformly convex domain in $R^{n}$, with $\p\Omega\in C^{3,1}$, $\varphi \in W^{4,p}(\Omega), \psi\in W^{2,p}(\Omega)$, $\inf_{\Omega}\psi>0$ and $f\in L^{p}(\Omega)$, $f\leq 0.$ Then, for uniformly convex solution $u\in C^{4}(\overline{\Omega})$ of (\ref{AMCE})-(\ref{SBV}), we have the estimate
\begin{equation}
\label{global-est}
\norm{u}_{W^{4,p}(\Omega)}\leq C,
\end{equation}
where $C$ depends on $n, p, \theta, \p\Omega, \Omega$, $\norm{f}_{L^{p}(\Omega)}$, $\norm{\varphi}_{W^{4,p}(\Omega)}, \norm{\psi}_{W^{2,p}(\Omega)}$, and $\inf_{\Omega} \psi.$
\end{thm}
The integrability of $f$ in Theorem \ref{mainthm} is optimal for the global estimate (\ref{global-est}).\\
By using degree theory as in \cite{TW2}, we obtain the following result as a consequence of Theorem \ref{mainthm}.
\begin{thm}
Let $p>n$ and let $\Omega$ be a uniformly convex domain in $R^{n}$, with $\p\Omega\in C^{3,1}$, $\varphi \in W^{4,p}(\Omega), \psi\in W^{2,p}(\Omega)$, $\inf_{\Omega}\psi>0$ and $f\in L^{p}(\Omega)$, $f\leq 0.$ Then, there is a unique uniformly convex solution $u\in W^{4,p}(\Omega)$ to the boundary value problem (\ref{AMCE})-(\ref{SBV}).
\end{thm}

For (\ref{AMCE}), the condition $f\leq 0$ in the case of $\theta=\frac{1}{n+2}$ corresponds to nonnegative affine mean curvature. Our theorem also covers the case $\theta =0$ which was also treated by Zhou \cite{Zh}. In \cite{Zh}, $f$ is bounded and $w$ is constant near the boundary $\p\Omega$. When $\theta=0$, (\ref{AMCE}) is known as Abreu's equation \cite{Ab} in the context of existence of K\"ahler metric of constant scalar curvature \cite{D1, D2, D3, D4}. \\

 Our proof of Theorem \ref{mainthm} is based on 
\begin{myindentpar}{1cm}
(i) the global $C^{2,\alpha}$ estimates for the Monge-Amp\`ere equation \cite{TW2} when the Monge-Amp\`ere measure is only assumed to be H\"older continuous, and \\
(ii) the global H\"older continuity estimates for solutions of the linearized Monge-Amp\`ere equations which are of independent interest.
\end{myindentpar}
We state these global H\"older estimates in the following theorem.
\begin{thm}
\label{global-h}
Let $\Omega$ be a bounded, uniformly convex domain in $\R^{n}$ with $\p\Omega\in C^{3}$. Let $u: \overline{\Omega}\rightarrow \R$, $u\in C^{0,1}(\overline{\Omega})\cap C^{2}(\Omega)$  be a convex function satisfying
$$0<\lambda\leq \det D^{2}u\leq \Lambda<\infty$$
and
$$u\mid_{\p\Omega}\in C^{3}.$$
 Let $v$ be the continuous solution to the linearized Monge-Amp\`ere equation $$U^{ij} v_{ij} = g\quad \text{in}\quad \Omega; v=\varphi\quad \text{on}\quad \partial\Omega.$$ Assume that
$$\|\varphi\|_{L^{\infty}(\Omega)} + \|g\|_{L^{n}(\Omega)}\leq K$$
and $\varphi$ is $C^{\alpha}$ on $\partial\Omega$, namely,
$$|\varphi(x)-\varphi(y)|\leq L |x-y|^{\alpha}~\text{for}~x, y\in\partial\Omega.$$ Then, $v\in C^{\beta}(\overline{\Omega})$ where $\beta$ depends only on $\lambda, \Lambda, n, K, L, \alpha$, $diam (\Omega)$, and the uniform convexity of $\Omega.$ 
\end{thm}
\begin{rem}
Several remarks are in order.
\begin{myindentpar}{1cm} 
(1) Theorem \ref{global-h} is the global counterpart of Caffarelli-Guti\'errez's interior H\"older estimates for solutions to the linearized Monge-Amp\`ere equations \cite{CG}.\\
(2) If $g\in L^{\infty}(\Omega)$, $\varphi\in C^{1,1}(\partial\Omega)$ and $\det D^{2}u\in C(\overline{\Omega})$ then Theorem \ref{global-h} is a consequence of the global $C^{1,\alpha}$ estimates for solutions to the linearized Monge-Amp\`ere equations obtained in \cite{LS1}. See Theorem 2.5 and Remark 7.1 in that paper. The point of interest in Theorem \ref{global-h} lies in the fact that the data are less regular.
\end{myindentpar}
\end{rem}

Under the assumptions of Theorem \ref{global-h}, the linearized Monge-Amp\`ere operator $L_u:= U^{ij}\p_{i}\p_{j}$ is in general not uniformly elliptic, i.e., the eigenvalues of $U = (U^{ij})$ are not necessarily bounded away from $0$ and $\infty.$ Moreover, $L_u$ can be possibly singular near the boundary; even if $\det D^{2} u$ is constant in $\overline{\Omega}$, $U$ can blow up logarithmically at the boundary; see \cite[Proposition 2.6]{LS1}. As in \cite{LS1}, the degeneracy and singularity of $L_u$ are the main difficulties in proving Theorem \ref{global-h}. We handle the degeneracy of $L_u$ by working as in \cite{CG} with sections of solutions to the Monge-Amp\`ere equations. These sections have the same role as Euclidean balls have in the classical theory. To overcome the singularity of $L_u$ near the boundary, we use a Localization Theorem at the boundary for solutions to the Monge-Amp\`ere equations which was obtained by Savin in \cite{S,S2}.\\

We would like to comment on the difference between the method of the proof of Theorem \ref{mainTW} in \cite{TW2} and our method of the proof of Theorem \ref{mainthm}. To prove Theorem \ref{mainTW}, Trudinger and Wang first established the Lipschitz continuity of $w$ at the boundary $\p\Omega$; precisely
$$\abs{w(x)-w(x_{0})}\leq C\abs{x-x_{0}}~~\forall x\in \Omega, x_{0}\in\p\Omega.$$
They then used this property together with convexity analysis and perturbation arguments in the spirit of Caffarelli \cite{C2} for solutions of the Monge-Amp\`ere equation to conclude that the boundary sections of $u$ are of good shape; that is, each of these sections lies between two concentric balls whose radii ratio is under control. After that, they were able to apply the arguments of Caffarelli-Guti\'errez \cite{CG} for linearized Monge-Amp\`ere equation to conclude that $w$ is globally H\"older continuous. Finally, the global estimates for $u$ follows from their global $C^{2,\alpha}$ estimates for solutions of Monge-Amp\`ere equation \cite[Theorem 1.1]{TW2}.

Our method of the proof of Theorem \ref{mainthm} is slightly different. When $f$ is only assumed to be in $L^{p}$ $(p>n)$, we can only prove that $w$ is H\"older continuous at the boundary; precisely, for some $\alpha\in (0,1)$,
\begin{equation}\abs{w(x)-w(x_{0})}\leq C\abs{x-x_{0}}^{\alpha}~~\forall x\in \Omega, x_{0}\in\p\Omega.
\label{w-h}
\end{equation}
However, this is enough for us to establish the global H\"older continuity for $w$. To do this, we first use the Localization theorem for solutions of the Monge-Amp\`ere equation proved by Savin \cite{S} to obtain some mild control on the boundary sections of $u$. Precisely, each of these sections lies between two concentric balls whose radii ratio behaves like logarithm of the distance from their center to the boundary. These controls are {\it independent} of the boundary behavior of $w$. They just depend on the bounds on $w$ and information of $u$ on the boundary. The use of the Localization theorem in combination with Caffarelli-Guti\'errez's interior H\"older estimates gives us global H\"older continuity of $w$. \\

In this paper, we denote by $c, C, C_{0}, C_{1}, C_{2}, C_{n}, etc...$ positive constant depending on structural information, such as $n, p, \theta, \p\Omega, \Omega$, $\norm{f}_{L^{p}(\Omega)}, \norm{\varphi}_{W^{4,p}(\Omega)}, \norm{\psi}_{W^{2,p}(\Omega)}$, and $\inf_{\Omega} \psi$ and their values may change from line to line whenever there is no possibility of confusion. We refer to such constants as {\it universal constants}.\\

The rest of the paper is organized as follows. In Section \ref{global-CG}, we establish global H\"older continuity for solutions to the linearized Monge-Amp\`ere equations, thus proving Theorem \ref{global-h}. In Section \ref{global-AMCE}, we prove Theorem \ref{mainthm} about the global second derivative estimates for the second boundary value problem of the prescribed affine mean curvature equation.

{\bf Acknowledgments.} {\it The author would like to thank Neil Trudinger for his interest, support and encouragement during the preparation of this paper. The author also acknowledges the financial support and hospitality of the Center of Mathematics and its Applications, the Australian National University, Canberra, Australia where this work was completed during his stay there as a Visiting Fellow.}

\section{Global H\"older continuity for solutions to the linearized Monge-Amp\`ere equations}
\label{global-CG}
In this section, we will prove Theorem \ref{global-h}. Throughout, we assume that $\Omega$ is a uniformly convex domain. 
\subsection{Boundary H\"older continuity for solutions of non-uniformly elliptic equations}
The result in this subsection establishes boundary H\"older continuity for solutions to non-uniformly elliptic, linear equations without lower order terms. It is a refinement of Proposition 4.14 in \cite{CC}; see also Lemma 4.3 in \cite{GN1}. It states as follows.
\begin{prop}\label{global-holder}
Let $v$ be the continuous solution to the equation $Lv:= a_{ij} v_{ij} = g$ in $\Omega$ with $v=\varphi$ on $\partial\Omega$. Here the matrix $(a_{ij})$ is assumed to be measurable, positive definite and satisfies $\det (a_{ij})\geq \lambda.$ We assume that
$$\|\varphi\|_{L^{\infty}(\Omega)} + \|g\|_{L^{n}(\Omega)}\leq K$$
and $\varphi$ is $C^{\alpha}$ on $\partial\Omega$, namely,
$$|\varphi(x)-\varphi(y)|\leq L |x-y|^{\alpha}~\text{for}~x, y\in\partial\Omega.$$ Then, there exist $\delta, C$ depending only on $\lambda, n, K, L, \alpha$, $diam (\Omega)$, and the uniform convexity of $\Omega$ so that, for any $x_{0}\in\partial\Omega$, we have
$$|v(x)-v(x_{0})|\leq C|x-x_{0}|^{\frac{\alpha}{\alpha +2}},~\forall x\in \Omega\cap B_{\delta}(x_{0}). $$
\end{prop}
\begin{proof}
Without loss of generality, we assume that $K>1$, $\lambda =1$ and
$$\Omega\subset \R^{n}\cap \{x_{n}>0\},~0\in\p\Omega.$$
Take $x_{0} =0$. 
By the Aleksandrov-Bakelman-Pucci (ABP) estimate \cite{GT} and the assumption $\det (a_{ij})\geq 1$, we have
$$|v(x)|\leq \sup_{\Omega}\|\varphi\|_{L^{\infty}(\Omega)} + C_{n}diam (\Omega) \|g\|_{L^{n}(\Omega)}\leq C_{0}K~\forall~ x\in \Omega$$
and hence, for any $\varepsilon \in (0,1)$
\begin{equation}|v(x)-v(0)\pm \e|\leq 3C_{0}K:= C_{1}.
\label{gen-ineq}
\end{equation}
Consider now the functions
$$h_{\pm}(x) := v(x)- v(0)\pm \e\pm C_{1} (inf \{y_{n}: y\in \overline{\Omega}\cap\partial B_{\delta_{2}}(0)\})^{-1} x_{n}$$
in the region $$A:= \Omega\cap B_{\delta_{2}}(0)$$ where $\delta_{2}$ is small to be chosen later.\\
Note that, if $x\in\partial \Omega$ with $$|x|\leq \delta_{1}(\e):= (\e/L)^{1/\alpha}$$ then, we have
\begin{equation}
\label{bdr-ineq}|v(x)-v(0)| =|\varphi(x)-\varphi(0)| \leq L|x|^{\alpha} \leq \e.
\end{equation}
It follows that, if we choose $\delta_{2}\leq \delta_{1}$ then from (\ref{gen-ineq}) and (\ref{bdr-ineq}), we have
$$h_{-}\leq 0, h_{+}\geq 0~\text{on}~\partial A.$$
On the other hand,
$$Lh_{\pm}= g~\text{in}~A.$$
The ABP estimate applied in $A$ gives
$$h_{-}\leq  C_{n}diam (A) \|g\|_{L^{n}(A)}\leq C_{n}K\delta_{2}~\text{in}~ A$$
and $$
h_{+}\geq - C_{n}diam (A) \|g\|_{L^{n}(A)}\geq  -C_{n}K\delta_{2}~\text{in}~ A.$$
By restricting $\e\leq \left(\frac{L}{(C_{n}K)^{\alpha}}\right)^{\frac{1}{1-\alpha}}$, we can assume that
$$\delta_{1} = (\e/L)^{1/\alpha}\leq \frac{\e}{C_{n}K}.$$
Then, for $\delta_{2}\leq \delta_{1}$, we have $C_{n}K\delta_{2}\leq \e$ and thus, for all $x\in A$, we have
$$|v(x)-v(0)|\leq 2\e + C_{1} (inf \{y_{n}: y\in \overline{\Omega}\cap\partial B_{\delta_{2}}(0)\})^{-1} x_{n}.$$
The uniform convexity of $\Omega$ gives
\begin{equation}
inf \{y_{n}: y\in \overline{\Omega}\cap\partial B_{\delta_{2}}(0)\} \geq C_{2}^{-1}\delta^2_{2}.
\end{equation}
Therefore, choosing $\delta_{2}= \delta_{1}$, we obtain
$$|v(x)-v(0)|\leq 2\e + C_{1} (inf \{y_{n}: y\in \overline{\Omega}\cap\partial B_{\delta_{2}}(0)\})^{-1} x_{n}= 2\e + \frac{2C_{1}C_{2}}{\delta_{2}^2}x_{n}~\text{in}~ A.$$
As a consequence, we have just obtained the following inequality
\begin{equation}
\label{op-ineq}
|v(x)-v(0)|\leq 2\e + \frac{2C_{1}C_{2}}{\delta_{2}^2}|x| = 2\e + 2C_{1}C_{2}L^{2/\alpha}\e^{-2/\alpha}|x|
\end{equation}
for all $x,\e$ satisfying the following conditions
\begin{equation}
\label{xe-ineq}
|x|\leq \delta_{1}(\e):= (\e/L)^{1/\alpha}, \e\leq \left(\frac{L}{(C_{n}K)^{\alpha}}\right)^{\frac{1}{1-\alpha}}: = c_{1}(\alpha, L, K, n).
\end{equation}
Finally, let us choose 
$$\e = |x|^{\frac{\alpha}{\alpha + 2}}.$$
It satisfies the conditions in (\ref{xe-ineq}) if $c_{1}\geq |x|^{\frac{\alpha}{\alpha + 2}}\geq L|x|^{\alpha}$, or 
$$|x|\leq \min\{c_{1}^{\frac{\alpha +2}{\alpha}}, L^{-\frac{\alpha +2}{\alpha^2}}\}:=\delta.$$
Then, by (\ref{op-ineq}), we have for all $x\in \Omega\cap B_{\delta}(0)$
$$|v(x)-v(0)| \leq C|x|^{\frac{\alpha}{\alpha + 2}}$$
where $C= 2 + 2C_{1}C_{2}L^{2/\alpha}.$
\end{proof}

The above proposition, applied to $a_{ij}= U^{ij}$ where $U = (U^{ij})$ is the matrix of cofactors of the Hessian matrix $D^2 u$ of the convex function $u$ satisfying
$$\lambda\leq \det D^{2}u\leq \Lambda,$$
gives the boundary continuity for solutions to the linearized Monge-Amp\`ere equation
$$U^{ij} v_{ij} =g.$$ 
This combined with the interior H\"older continuity estimates of Caffarelli-Guti\'errez \cite{CG} gives the global H\"older estimates for solutions to the linearized Monge-Amp\`ere equations on uniformly convex domains. The proof follows the same lines as in the proofs of Theorem 2. 5 and Proposition 2.6 in \cite{LS1}. The rest of this section will be devoted to the proof of these global H\"older estimates.

\subsection{The Localization Theorem }
In this subsection, we state the main tool used in the proof of Theorem \ref{global-h}, the localization theorem.\\  Let $\Omega\subset \R^{n}$ be a bounded convex set with
\begin{equation}\label{om_ass}
B_\rho(\rho e_n) \subset \, \Omega \, \subset \{x_n \geq 0\} \cap B_{\frac 1\rho},
\end{equation}
for some small $\rho>0$. Assume that 
\begin{equation}
\Omega~ \text{contains an interior ball of radius $\rho$ tangent to}~ \p 
\Omega~ \text{at each point on} ~\p \Omega\cap\ B_\rho.
\label{tang-int}
\end{equation}
Let $u : \overline \Omega \rightarrow \R$, $u \in C^{0,1}(\overline 
\Omega) 
\cap 
C^2(\Omega)$  be a convex function satisfying
\begin{equation}\label{eq_u}
\det D^2u =f, \quad \quad 0 <\lambda \leq f \leq \Lambda \quad \text{in $\Omega$},
\end{equation} and 
assume that
\begin{equation}\label{0grad}
u(0)=0, \quad \nabla u(0)=0.
\end{equation}
Let $S_{h}(u)$ be the section of $u$ at $0$ with level $h$:
$$S_h := \{x \in \overline \Omega : \quad u(x) < h \}.$$

If the boundary data has quadratic growth near $\{x_n=0\}$ then, as $h \rightarrow 0$, $S_h$ is equivalent to a half-ellipsoid centered at 0. This is the content
of the Localization Theorem proved in \cite{S,S2}. Precisely, this theorem reads as follows.

\begin{thm}[Localization Theorem \cite{S,S2}]\label{main_loc}
 Assume that $\Omega$ satisfies \eqref{om_ass}-(\ref{tang-int}) and $u$ satisfies 
\eqref{eq_u}, 
\eqref{0grad} above and,
\begin{equation}\label{commentstar}\rho |x|^2 \leq u(x) \leq \rho^{-1} 
|x|^2 \quad \text{on $\p \Omega \cap \{x_n \leq \rho\}.$}\end{equation}
Then, for each $h<k$ there exists an ellipsoid $E_h$ of volume $\omega_{n}h^{n/2}$ 
such that
$$kE_h \cap \overline \Omega \, \subset \, S_h \, \subset \, k^{-1}E_h \cap \overline \Omega.$$

Moreover, the ellipsoid $E_h$ is obtained from the ball of radius $h^{1/2}$ by a
linear transformation $A_h^{-1}$ (sliding along the $x_n=0$ plane)
$$A_hE_h= h^{1/2}B_1,\quad \det A_{h} =1,$$
$$A_h(x) = x - \tau_h x_n, \quad \tau_h = (\tau_1, \tau_2, \ldots, 
\tau_{n-1}, 0), $$
with
$$ |\tau_{h}| \leq k^{-1} |\log h|.$$
The constant $k$ above depends only on $\rho, \lambda, \Lambda, n$.
\end{thm}

 The ellipsoid $E_h$, or equivalently the linear map $A_h$, 
provides useful information about the behavior of $u$ 
near the origin. From Theorem \ref{main_loc} we also control the shape of sections that are tangent to $\p \Omega$ at the origin. Before we state this result we introduce the notation for the section of $u$ centered at $x\in \overline \Omega$ at height $h$:
\begin{equation*}
 S_{x,h} (u) :=\{y\in \overline \Omega:  u(y) < u(x) + \nabla u(x) (y- x) +h\}.
\end{equation*}

\begin{prop}\label{tan_sec}
Let $u$ and $\Omega$ satisfy the hypotheses of the Localization Theorem \ref{main_loc} at the 
origin. Assume that for some $y \in \Omega$ the section $S_{y,h} \subset \Omega$
is tangent to $\p \Omega$ at $0$ for some $h \le c$ with $c$ universal. Then there exists a small 
 constant $k_0>0$ depending on $\lambda$, $\Lambda$, $\rho $ and $n$ such that
$$ \nabla u(y)=a e_n 
\quad \mbox{for some} \quad   a \in [k_0 h^{1/2}, k_0^{-1} h^{1/2}],$$
$$k_0 E_h \subset S_{y,h} -y\subset k_0^{-1} E_h, \quad \quad k_0 h^{1/2} \le dist(y,\p \Omega) \le k_0^{-1} h^{-1/2}, \quad $$
with $E_h$ the ellipsoid defined in the Localization Theorem \ref{main_loc}.
\end{prop}

Proposition \ref{tan_sec} is a consequence of Theorem \ref{main_loc} and was proved \cite{S3}. 
\subsection{Quadratic separation on the boundary}
The quadratic separation from tangent planes on the boundary for solutions to the Monge-Amp\`ere equation is a crucial assumption in the Localization Theorem \ref{main_loc}. This is the case for $u$ in Theorem \ref{global-h} as proved in \cite[Proposition 3.2]{S2}.
\begin{prop}
Let $u$ be as in Theorem \ref{global-h}. Then, on $\p \Omega$, 
$u$ separates quadratically from its tangent planes on $\p \Omega$. This means that
if $x_0 \in 
\p \Omega$ then
\begin{equation}
 \rho\abs{x-x_{0}}^2 \leq u(x)- u(x_{0})-\nabla u(x_{0}) (x- x_{0}) \leq 
\rho^{-1}\abs{x-x_{0}}^2,
\label{eq_u1}
\end{equation}
for all $x \in \p\Omega,$ for some small constant $\rho$ universal. 
\label{quadsep}
\end{prop}

When $x_0 \in \p \Omega,$ the term $\nabla u(x_0)$ is understood in the 
sense that
$$ x_{n+1}=u(x_0)+\nabla u(x_0) \cdot (x-x_0) $$ is a supporting 
hyperplane for the graph of $u$ but for any $\eps >0$,
$$ x_{n+1}=u(x_0)+(\nabla u(x_0)- \eps \nu_{x_0}) \cdot (x-x_0) $$
is not a supporting hyperplane, where $\nu_{x_0}$ denotes the exterior 
unit normal to $\p \Omega$ at $x_0$. In fact we show in \cite[Proposition 4.1]{LS1}
that our hypotheses imply that $u$ is 
always differentiable at $x_0$ and then $\nabla u(x_0)$ is defined also in 
the classical sense.

\subsection{Global H\"older continuity for solutions of the linearized Monge-Amp\`ere equation} We are now ready to prove Theorem \ref{global-h}.
\begin{proof}[Proof of Theorem \ref{global-h}] We recall from Proposition \ref{quadsep} that $u$ separates quadratically from its tangent planes on $\p\Omega$. Therefore, Proposition \ref{tan_sec} applies.
Let $y\in \Omega $ with $$r:=dist (y,\partial\Omega) \le c,$$ for $c$ universal, and consider the maximal section $S_{y,\bar{h}(y)}$ centered at $y$, i.e.,
$$\bar{h}(y)=max\{h\,| \quad S_{y,h}\subset \Omega\}.$$
When it is clear from the context, we write $\bar{h}$ for $\bar{h}(y)$.
By Proposition \ref{tan_sec} applied at the point $$x_0\in \p S_{y,\bar h} \cap \p \Omega,$$ we have
 \begin{equation}\bar h^{1/2} \sim r,
\label{hr}
\end{equation}
and $S_{y,\bar h}$ is equivalent to an ellipsoid $E$ i.e
$$cE \subset S_{y, \bar h}-y \subset CE,$$
where
\begin{equation}E :=\bar h^{1/2}A_{\bar{h}}^{-1}B_1, \quad \mbox{with} \quad \|A_{\bar{h}}\|, \|A_{\bar h}^{-1} \| \le C |\log \bar h|; \det A_{\bar{h}}=1.
\label{eh}
\end{equation}
We denote $$u_y:=u-u(y)-\nabla u(y) (x-y).$$
The rescaling $\tilde u: \tilde S_1 \to \R$ of $u$ 
$$\tilde u(\tilde x):=\frac {1}{ \bar h} u_y(T \tilde x) \quad \quad x=T\tilde x:=y+\bar h^{1/2}A_{\bar{h}}^{-1}\tilde x,$$
satisfies
$$\det D^2\tilde u(\tilde x)=\tilde f(\tilde x):=f(T \tilde x),  $$
and
\begin{equation}
\label{normalsect}
B_c \subset \tilde S_1 \subset B_C, \quad \quad \tilde S_1=\bar h^{-1/2} A_{\bar h}(S_{y, \bar h}- y),
\end{equation}
where $\tilde S_1$ represents the section of $\tilde u$ at the origin at height 1.

We define also the rescaling $\tilde v$ for $v$
$$\tilde v(\tilde x):= v(T\tilde x)- v(x_{0}),\quad \tilde x\in \tilde S_{1}.$$
Then $\tilde v$ solves
$$\tilde U^{ij} \tilde v_{ij} = \tilde g(\tilde x):= \bar{h} g(T\tilde x).$$
Now, we apply Caffarelli-Guti\'errez's interior H\"older estimates \cite{CG, TW3} to $\tilde v $ to obtain
$$\abs{\tilde v (\tilde z_{1})-\tilde v(\tilde z_{2})}\leq C\abs{\tilde z_{1}-\tilde z_{2}}^{\beta} \{\norm{\tilde v }_{L^{\infty}(\tilde S_{1})} + \norm{\tilde g}_{L^{n}(\tilde S_{1})}\},\quad\forall \tilde z_{1}, \tilde z_{2}\in \tilde S_{1/2},$$
for some small constant $\beta\in (0,1)$ depending only on $n, \lambda, \Lambda$.\\
By (\ref{normalsect}), we can decrease $\beta$ if necessary and thus we can assume that
$$2\beta\leq \frac{\alpha}{\alpha + 2}: =2\gamma.$$ Note that, by (\ref{eh})
$$ \norm{\tilde g}_{L^{n}(\tilde S_{1})} = \bar{h}^{1/2}\norm{g}_{L^{n}(S_{y, \bar{h}})}.$$
We observe that (\ref{hr}) and (\ref{eh}) give
$$B_{C r\abs{log r}}(y)\supset S_{y,\bar{h}} \supset S_{y,\bar{h}/2}\supset B_{c\frac{r}{\abs{log r}}}(y)$$
and
$$diam (S_{y,\bar{h}})\leq Cr\abs{log r}.$$
By Proposition \ref{global-holder}, we have
$$\norm{\tilde v }_{L^{\infty}(\tilde S_{1})} \leq C diam (S_{y, \bar{h}})^{2\gamma} \leq C (r\abs{\log r})^{2\gamma}.$$
Hence
$$\abs{\tilde v (\tilde z_{1})-\tilde v(\tilde z_{2})}\leq C\abs{\tilde z_{1}-\tilde z_{2}}^{\beta}\{(r\abs{\log r})^{2\gamma}  + \bar{h}^{1/2}\norm{g}_{L^{n}(S_{y, \bar{h}})}\}~\forall \tilde z_{1}, \tilde z_{2}\in \tilde S_{1/2}.$$
 Rescaling back and using
$$\tilde z_1-\tilde z_2=\bar h^{-1/2}A_{\bar h}(z_1-z_2),$$
and the fact that
$$\abs{\tilde z_1-\tilde z_2}\leq \norm{\bar h^{-1/2}A_{\bar h}}\abs{z_1-z_2} \leq C \bar{h}^{-1/2}\abs{\log \bar{h}}\abs{z_1-z_2}\leq
C r^{-1}\abs{log r}\abs{z_1-z_2},$$
we find
\begin{equation}|v(z_1)-v( z_2)|  \le  |z_1-z_2|^{\beta} \quad \forall  z_1, z_2 \in  S_{y,\bar h/2} .
\label{oscv}
\end{equation}
Notice that this inequality holds also in the Euclidean ball $B_{c\frac{r}{\abs{log r}}}(y)\subset S_{y,\bar h/2}$. Combining this with Proposition \ref{global-holder}, we easily obtain that $$[v]_{C^\beta(\bar \Omega)} \le C,$$ for some $\beta\in (0,1)$, $C$ universal.\\ For completeness, we include the details. By rescaling the domain, we can assume that
$$\Omega\subset B_{1/100}(0).$$
We estimate, for $x$ and $y$ in $\Omega$
$$\frac{\abs{v(x)-v (y)}}{\abs{x-y}^{\beta}}.$$
Let $r_{x} = dist(x, \p\Omega)$ and $r_{y}= dist (y, \p\Omega).$ Suppose that $r_{y}\leq r_{x},$ say, and take $x_{0}\in\p\Omega$ and $ y_{0}\in \p\Omega$ such that $r_{x}= \abs{x- x_{0}}$ and $r_{y} = \abs{y-y_{0}}.$ From the interior H\"older estimates of Caffarelli-Guti\'errez, we only need to consider the case $r_{y}\leq r_{x}\leq c.$\\
Assume first that
$$\abs{x-y}\leq c \frac{r_{x}}{\abs{log r_{x}}}.$$
Then $y\in B_{c \frac{r_{x}}{\abs{log r_{x}}}}(x)\subset S_{x,\bar{h}(x)/2}.$ By (\ref{oscv}), we have
$$\frac{\abs{v(x)-v (y)}}{\abs{x-y}^{\beta}}\leq 1.$$
Assume finally that
\begin{equation*}
\label{bigxy}
\abs{x-y}\geq c \frac{r_{x}}{\abs{log r_{x}}}.
\end{equation*}
We claim that 
$$r_{x}\leq C\abs{x-y}\abs{log \abs{x-y}}.$$
Indeed, if 
$$1>r_{x}\geq \abs{x-y}\abs{log \abs{x-y}}\geq \abs{x-y} $$
then
$$r_{x}\leq \frac{1}{c}\abs{x-y}\abs{log r_{x}}\leq \frac{1}{c}\abs{x-y}\abs{log \abs{x-y}}.$$
Now, we have
$$\abs{x_0-y_0}\leq r_{x} + \abs{x-y} + r_{y}\leq C \abs{x-y}\abs{log \abs{x-y}}.$$
Hence, by Proposition \ref{global-holder} and recalling $2\gamma =\frac{\alpha}{\alpha + 2},$
\begin{eqnarray*}\abs{v(x)-v (y)}&\leq& \abs{v(x)- v(x_{0})} + \abs{v(x_{0})- v(y_{0})} + \abs{v(y_{0})- v(y)}\\ & \leq& C \left(r_{x}^{2\gamma} + \abs{x_{0}- y_{0}}^{\alpha} + r_{y}^{2\gamma}\right) \\ &\leq& C\left(\abs{x-y}\abs{log \abs{x-y}}\right)^{2\gamma}\leq C \abs{x-y}^{\beta}.
\end{eqnarray*}

\end{proof}
\section{Global second derivative estimates for the prescribed affine mean curvature equation}
\label{global-AMCE}
In this section, we prove Theorem \ref{mainthm}.
Let $u\in C^{4}(\overline{\Omega})$ be the uniformly convex solution of (\ref{AMCE})-(\ref{SBV}).
We first establish bounds on the determinant $\det D^2 u$ via those of $w$. 
\begin{lem}
There exists a constant $C>0$ depending only on $n, p, \theta, \Omega$, $\norm{f}_{L^{p}(\Omega)}$, $\norm{\psi}_{W^{2,p}(\Omega)}$, and $\inf_{\Omega} \psi$ such that any uniformly convex solution $u\in C^{4}(\overline{\Omega})$ of (\ref{AMCE})-(\ref{SBV}) satisfies
\begin{equation}
\label{wbound}
C^{-1}\leq w\leq C.
\end{equation}
\label{wbound-lem}
\end{lem}
\begin{proof}
Because $f\leq 0$, by the maximum principle, $w$ attains its minimum on the boundary. Since $\inf_{\Omega}\psi>0$, we obtain the first inequality in (\ref{wbound}).\\
 For the upper bound of $w$, we use (\ref{wu}) in the form $\det D^{2} u = w^{\frac{1}{\theta-1}}$ and the Aleksandrov-Bakelman-Pucci estimate. By this estimate, we have
\begin{eqnarray*}
\norm{w}_{L^{\infty}(\Omega)}& \leq& \norm{\psi}_{L^{\infty}(\p\Omega)} + C\norm{\frac{f}{(\det U)^{1/n}}}_{L^{n}(\Omega)}\\
&= & \norm{\psi}_{L^{\infty}(\p\Omega)} +  C\norm{\frac{f}{(\det D^2u)^{\frac{n-1}{n}}}}_{L^{n}(\Omega)}\\
&=& \norm{\psi}_{L^{\infty}(\p\Omega)}  + C\norm{f w^{\frac{n-1}{n(1-\theta)}}}_{L^{n}(\Omega)}\\
&\leq& \norm{\psi}_{L^{\infty}(\p\Omega)} + C\norm{w}^{\frac{n-1}{n(1-\theta)}}_{L^{\infty}(\Omega)}\norm{f}_{L^{n}(\Omega)}.
\end{eqnarray*}
Because $\theta<1/n$, we have 
$$\frac{n-1}{n(1-\theta)}<1$$
and the second inequality in (\ref{wbound}) follows.
\end{proof}
We are now ready to complete the proof of Theorem \ref{mainthm}.
\begin{proof}[Proof of theorem \ref{mainthm}] By Lemma \ref{wbound-lem}, 
$$C^{-1}\leq \det D^2 u \leq C.$$
Note that, by (\ref{AMCE}), $w$ is the solution to the linearized Monge-Amp\`ere equation $U^{ij}w_{ij} =f$ with boundary data $w=\psi.$ Because $\psi\in W^{2,p}(\Omega)$ with $p>n$, $\psi$ is clearly H\"older continuous on $\partial\Omega$. 
Thus, by Theorem \ref{global-h}, $w$ is H\"older continuous up the boundary. Rewriting (\ref{wu}) as
$$\det D^2 u = w ^{\frac{1}{\theta-1}},$$
and noticing $u=\varphi$ on $\p\Omega$ where $\varphi\in C^{3}(\overline{\Omega})$, we obtain
$u\in C^{2,\alpha}(\overline{\Omega})$ \cite[Theorem 1.1]{TW2}. Thus (\ref{AMCE}) is a uniformly elliptic, second order partial differential equations in $w$. Hence $w\in W^{2,p}(\Omega)$ and in turn $u\in W^{4, p}(\Omega)$ with desired estimate
\begin{equation}
\norm{u}_{W^{4,p}(\Omega)}\leq C,
\end{equation}
where $C$ depends on $n, p, \theta, \p\Omega, \Omega$, $\norm{f}_{L^{p}(\Omega)}, \norm{\varphi}_{W^{4,p}(\Omega)}, \norm{\psi}_{W^{2,p}(\Omega)}$, and $\inf_{\Omega} \psi.$
\end{proof}


\begin{thebibliography}{9999}
\bibitem[Ab]{Ab} Abreu, M. K\"ahler geometry of toric varieties and extremal metrics, {\it Inter. J. Math.} {\bf 9} (1998), 641-651.
\bibitem[Bl]{Bl} Blaschke, W. Vorlesungen \"uber Differential geometrie, Berlin, 1923.
\bibitem[CC]{CC} Caffarelli, L. A., and ~Cabr\'e, X. {\em Fully nonlinear elliptic
equations.} American Mathematical Society Colloquium Publications,
volume 43, 1995.
    
\bibitem[C2]{C2} Caffarelli, L. A. Interior $W^{2,p}$ estimates for solutions of Monge-Amp\`ere equation, {\it Ann. of Math.} {\bf 131} (1990), 135-150.
\bibitem[CG]{CG} Caffarelli, L. A.; Guti\'{e}rrez, C. E. Properties of the solutions of the linearized Monge-Amp\`ere equation.  
{\it Amer. J. Math.}  {\bf 119}  (1997),  no. 2, 423--465.

\bibitem[D1]{D1}Donaldson, S. K. Scalar curvature and stability of toric varieties.  {\it J. Differential Geom.}  {\bf 62}  (2002),  no. 2, 289--349.
\bibitem[D2]{D2} Donaldson, S. K. Interior estimates for solutions of Abreu's equation.  {\it Collect. Math.}  {\bf 56}  (2005),  no. 2, 103--142
\bibitem[D3]{D3} Donaldson, S. K. Extremal metrics on toric surfaces: a continuity method.  {\it J. Differential Geom. } {\bf 79}  (2008),  no. 3, 389--432.
\bibitem[D4]{D4} Donaldson, S. K. Constant scalar curvature metrics on toric surfaces.  {\it Geom. Funct. Anal.}  {\bf 19}  (2009),  no. 1, 83--136.
\bibitem[GT]{GT} Gilbarg, D.; Trudinger, N. S. {\it Elliptic partial 
differential equations of second order}. Reprint of the 1998 edition. Classics in Mathematics. Springer-Verlag, Berlin, 2001.
\bibitem[GN1]{GN1} Guti\'errez, C.; Nguyen, T. Interior gradient estimates for solutions to the linearized Monge-Amp\`ere equations, 
{\it Adv. Math.} {\bf 228} (2011), 2034-2070.
\bibitem[LS1]{LS1} Le, N. Q.; Savin, O. Boundary regularity for solutions to the linearized Monge-Amp\`ere equations, {\it Preprint}, arXiv:1109.5677v1 [math.AP].
\bibitem[NS]{NS} Nomizu, K. and Sasaki, T. {\it Affine differential geometry}, Cambridge University Press, 1994
\bibitem[S]{S} Savin, O. A localization property at the boundary for the Monge-Amp\`ere equation. arXiv:1010.1745v2 [math.AP].
\bibitem[S2]{S2} Savin, O. Pointwise $C^{2,\alpha}$ estimates at the boundary for the Monge-Amp\`ere equation. arXiv:1101.5436v1 [math.AP]
\bibitem[S3]{S3} Savin, O. Global $W^{2,p}$ estimates for the Monge-Amp\`ere equation. arXiv:1103.0456v1 [math.AP].
\bibitem[Si]{Si} Simon, U. Affine differential geometry, in {\it Handbook of differential geometry}, North-Holland, Amsterdam, 2000, 905--961.
\bibitem[TW]{TW} Trudinger, N. S.; Wang, X. J. The Bernstein problem for affine maximal hypersurfaces.  {\it Invent. Math.}  {\bf 140}  (2000),  no. 2, 399--422.
\bibitem[TW1] {TW1}
Trudinger, N.S. and Wang, X.J., The affine plateau problem, {\it J. Amer.
Math. Soc.} {\bf 18}(2005), 253-289.
\bibitem[TW2]{TW2} Trudinger N.S., Wang X.J, Boundary regularity for Monge-Amp\`ere and affine maximal surface equations, {\it Ann. of Math.} {\bf 167} (2008), 993-1028.
\bibitem[TW3]{TW3} Trudinger, N. S.; Wang, X. J. The Monge-Amp\`{e}re equation and its 
geometric applications.  {\it Handbook of geometric analysis.} No. 1,  467--524, Adv. Lect. Math. (ALM), 7, Int. Press, Somerville, MA, 2008.
\bibitem[Zh]{Zh} Zhou, B. The first boundary value problem for Abreu's equation, {\it preprint}, arXiv:1009.1834v1 [math.AP].
\end{thebibliography}
\end{document}